\documentclass[a4paper,12pt]{article}
\usepackage{times, url}
\usepackage{graphicx,amsmath,amssymb,amsthm,amsfonts,bm,float,cite}
\newtheorem{theorem}{Theorem}[section]

\newtheorem{lemma}{Lemma}[section]
\newtheorem{definition}{Definition}[section]

\usepackage[none]{hyphenat}
\usepackage[center]{caption}
\usepackage[pdfencoding=auto]{hyperref}
\usepackage{bookmark}
 
\usepackage{mathtools}
\DeclarePairedDelimiter\ceil{\lceil}{\rceil}

\begin{document}

\begin{center}
{\LARGE \bf Solving the $\bm{106}$ years old $\bm{3^k}$ points problem with the clockwise-algorithm}
\vspace{8mm}

{\Large \bf Marco Rip\`a}
\vspace{3mm}

World Intelligence Network \\ 
Rome, Italy \\
e-mail: \url{marcokrt1984@yahoo.it}
\vspace{2mm}
\end{center}

\noindent {\bf Abstract:} \sloppy In this paper, we present the clockwise-algorithm that solves the extension in $k$-dimensions of the infamous nine-dot problem, the well-known two-dimensional thinking outside the box puzzle. We describe a general strategy that constructively produces minimum length covering trails, for any $k \in \mathbb{N}-\{0\}$, solving the NP-complete $(3 \times 3 \times \cdots \times 3)$-point problem inside $3 \times 3 \times \cdots \times 3$ hypercubes. In particular, using our algorithm, we explicitly draw different covering trails of minimal length $h(k)=\frac{3^k-1}{2}$, for $k=3, 4, 5$. Furthermore, we conjecture that, for every $k \geq 1$, it is possible to solve the $3^k$-point problem with $h(k)$ lines starting from any of the $3^k$ nodes, except from the central one. Finally, we cover a $3 \times 3 \times 3$ grid with a tree of size $12$.
\vspace{5mm}


\section{Introduction} \label{sec:Intr}
The classic \textit{nine-dot puzzle} \cite{1, 2} is the well-known thinking outside the box challenge \cite{3, 4}, and it corresponds to the two-dimensional case of the general $3^k$-point problem (assuming $k=2$) \cite{5, 6, 7, 8}.

The statement of the $3^k$-point problem is as follows:

\noindent ``Given a finite set of $3^k$ points in $\mathbb{R}^k$, we need to visit all of them (at least once) with a polygonal chain that has the minimum number of line segments, $h(k)$, and we simply define the aforementioned line segments as \textit{lines}. In detail, let $G_k$ be a $3 \times 3 \times \cdots \times 3$ grid in ${\mathbb{N}_0}^k$, we are asked to join all the points of $G_k$ with a minimum (link) length covering trail $C:=C(k)$ \linebreak($C(k)$ represents any trail consisting of $h(k)$ lines), without letting one single line of $C$ go outside of a $3 \times 3 \times \cdots \times 3$, $k$-dimensional, axis-aligned bounding box (i.e., remaining inside a $4 \times 4 \times \cdots \times 4$ AABB in $\mathbb{R}^k$, which strictly contains $G_k$, and we call it \textit{box})''.

It is trivial to note that the formulation of our problem is equivalent to asking:

\noindent ``Which is the minimum number of turns ($h(k)-1$) to visit (at least once) all the points of the $k$-dimensional grid $G_k=\{(0,1,2) \times (0,1,2) \times \cdots \times (0,1,2)\}$ with a connected series of line segments (i.e., a possibly self-crossing polygonal chain allowed to turn at nodes and Steiner points)?'' \cite{9, 10}.

The goal of the present paper is to solve the $3^k$-point problem, for every positive integer $k$.

We introduce a general algorithm, that we name the \textit{clockwise-algorithm}, which produces optimal covering trails for the $3^k$-point problem. In particular, we show that $C(k)$ has $h(k)=\frac{3^k-1}{2}$ lines, answering the most spontaneous $106$ years old question that arose from the original Loyd's puzzle \cite{2}.

The aspect of the $3^k$-point problem that most amazed us, when we eventually solved it, is the central role of Loyd’s expected solution for the $k=2$ case. In fact, the clockwise-algorithm, able to solve the main problem in a $k$-dimensional space, is the natural generalization of the classic solution of the nine-dot puzzle.


\section{Solving the \texorpdfstring{$\bm{3^k}$-} -point problem} \label{sec:2}

The stated $3^k$-point optimization problem, especially for $k<4$, appears to have concrete applications in manufacturing, drone routing, cognitive psychology, and integrated circuits (VLSI design). Many suboptimal bounds have been proved for the NP-complete \cite{11} $3^k$-point problem under additional constraints (such as limiting the solutions to Hamiltonian paths or considering only rectilinear spanning paths \cite{5, 7, 12}), but (to the best of our knowledge) the $3^{k>3}$-point problem remains unsolved to the present day, and this paper provides its first exact solution \cite{13}.

\subsection{A tight lower bound} \label{sec:SUBSECTION 2.1}

Given the $3^k$-point problem as introduced in Section \ref{sec:Intr}, if we remove its constraint on the inside the box solutions, then we have that a lower bound is provided by Theorem \ref{Theorem 1}.

\begin{theorem} \label{Theorem 1} For every positive integer $k$, $h(k) \geq \frac{3^k-1}{2}$.
\end{theorem}

\begin{proof} If $k=1$, then it is necessary to spend (at least) one line to join the $3$ points.

Given $k=2$, we already know that the nine-dot problem cannot be solved with less than $4$ lines (see Reference \cite{14}, assuming $n=3$).

Let $k$ be greater than $2$. We invoke the proof of Theorem \ref{Theorem 1} in Reference \cite{13}, substituting $n_i=3$.

Thus, Equation (4) of the above-mentioned Reference \cite{13} can be rewritten as
\begin{equation} \label{eq1}
h_l(3_1, 3_2,\ldots,3_k )=\ceil*{\frac{3^k-1}{2}},
\end{equation}
which is an integer (since $3^k-1$ is always even).

Therefore, $h(k) \geq h_l(3_1, 3_2, \ldots, 3_k ) = \frac{3^k-1}{2}$ for any (strictly positive) natural number $k$.
\end{proof}

It is redundant to point out that Theorem \ref{Theorem 1} provides also a valid lower bound for any $3^k$-point (arbitrary) \textit{box-constrained} problem. The purpose of the next subsection is to show that this bound matches $h(k)$ for every $k$.

\subsection{The clockwise-algorithm} \label{sec:SUBSECTION 2.2}

To introduce the clockwise-algorithm, let us begin from the trivial case $k=1$. This means that we have to visit $3$ collinear points with a single line, remaining inside a unidimensional box that is $3$ units long.

One solution is shown in Figure \ref{fig:Figure_1}.

\begin{figure}[H]
\begin{center}
\includegraphics[width=\linewidth]{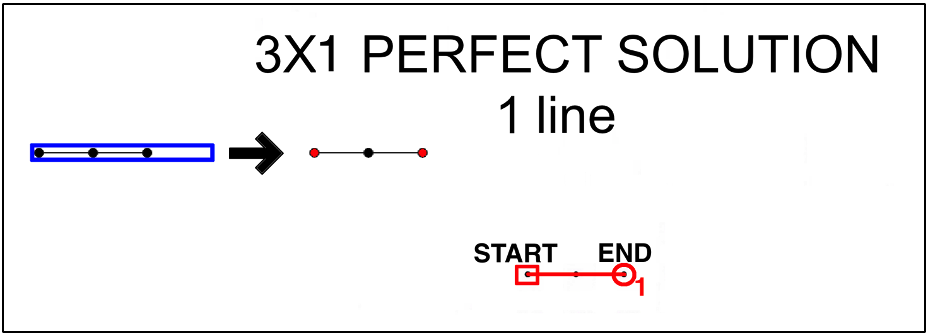}
\end{center}
\caption{Solving the $3 \times 1$ puzzle inside the box ($3$ units of length), starting from one of the line segment endpoints. The puzzle is solvable with this $C(1)$ path starting from both the red points.}
\label{fig:Figure_1}
\end{figure}

Considering the spanning path in Figure \ref{fig:Figure_1}, it is easy to see that we cannot solve the $3^1$-point problem starting from one point of $G_1$ if and only if this point is the central one.

Given $k=2$, we are facing the classic nine-dot puzzle considering a $3 \times 3$ box ($9$ units of area). The well-known Hamiltonian path shown in Figure \ref{fig:Figure_2} proves that we can solve the problem, without allowing any line to exit from the box, if we start from any node of $G_2$ except from the central one \cite{14}.

\begin{figure}[H]
\begin{center}
\includegraphics[width=\linewidth]{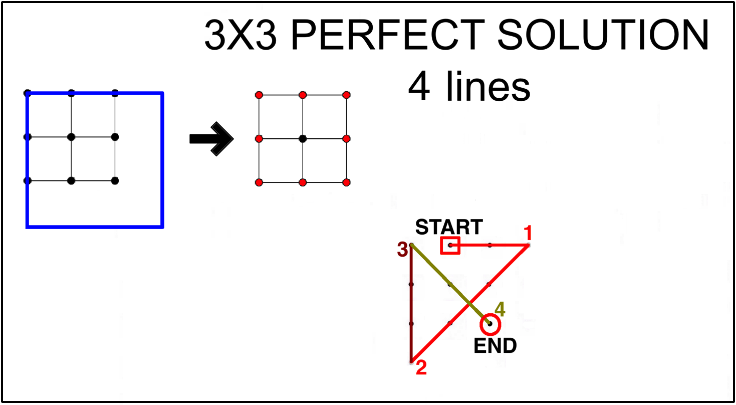}
\end{center}
\caption{$C(2)$ is a path that consists of $h(2)=\frac{3^2-1}{2}$ lines. In order to solve the $3 \times 3$ puzzle with $4$ lines starting from one node of $G_2$, it is necessary to avoid starting from the central point of the grid.}
\label{fig:Figure_2}
\end{figure}

Looking carefully at $C(2)$, as shown in Figure \ref{fig:Figure_2}, we note that line $1$ includes $C(1)$ if we simply extend it by one unit backward. Thus, $C(1)$ and the first line of $C(2)$ are essentially the same trail, and so they are considering the clockwise-algorithm. Line $2$ can be obtained from line $1$ going backward when we apply a standard rotation of $\frac{\pi}{4}$ radians: we are just spinning around in a two-dimensional space, forgetting the $3^{2-1}-1$ collinear points that will later be covered by the repetition of $C(1)$ following a different direction. We are now able to understand what line $3$ really is: it is just a link between the repeated $C(2-1)$ trail backward and the final $C(2-1)$ trail following the new direction. In general, the aforementioned link corresponds to line
$2 \cdot h(k-1)+1=3^{k-1}$ of any $C(k)$ generated by the clockwise-algorithm.

\begin{definition} \label{def1}
Let $G_3 \coloneqq \{(0,1,2) \times (0,1,2) \times (0,1,2)\}$. We call ``nodes'' all the $27$ points of $G_3$, as usual. In particular, we indicate the nodes $V_1 \equiv (0,0,0)$, $V_2 \equiv (2,0,0)$, $V_3 \equiv (0,2,0)$, $V_4 \equiv (0,0,2)$, $V_5 \equiv (2,2,0)$, $V_6 \equiv (2,0,2)$, $V_7 \equiv (0,2,2)$, $V_8 \equiv (2,2,2)$ as ``vertices'', we indicate the nodes $F_1 \equiv (1,1,0)$, $F_2 \equiv (1,0,1)$, $F_3 \equiv (0,1,1)$, $F_4 \equiv (2,1,1)$, $F_5 \equiv (1,2,1)$, $F_6 \equiv (1,1,2)$ as ``face-centes'', we call ``center'' the node $X_3 \equiv (1,1,1)$, and we indicate as ``edges'' the remaining $12$ nodes of $G_3$.
\end{definition}

Now, we are ready to describe the generalization of the original Loyd's covering trail to higher dimensions. Given $k=3$, a minimum length covering trail has already been shown in Reference \cite{13}, but this time we need to solve the problem inside a $3 \times 3 \times 3$ box. Our strategy is to follow the optimal two-dimensional covering trail (see Figure \ref{fig:Figure_2}) swirling in one more dimension, according to the $3$-step scheme given by lines $1$ to $3$ of $C(2)$, and beginning from a congruent starting point.

Thus, if we take one vertex of $G_3$, while we rotate in the space at every turn (as observed for $k=2$), it is possible to repeat twice (forward and backward) the whole $C(2)$ or, alternatively (Figure \ref{fig:Figure_3}), we can follow $\frac{8}{3}$ times the scheme provided by lines $1$ to $3$. In both cases, at the end of the process, $3^{3-2}-\frac{1}{3}$ gyratories have been performed, so we spend the $(3^{3-1})$-th line to close the sub-tour ($C(3)$ can never be a circuit plus we avoided extending its first line backward, but we have already seen that this fact does not really matter), joining $3-1$ new points. In this way, we reach the starting vertex again, and the last $3^3-1$ unvisited nodes belong only to $G_{k-1}=G_2$ (choosing the right direction). Therefore, we can finally paste $C(2)$ (Figure \ref{fig:Figure_2}) by extending one unit backward its first line (the new $(2 \cdot h(3-1)+2)$-th line) to visit all the $3^2$ nodes of $G_{3-1}$.

\begin{figure}[H]
\begin{center}
\includegraphics[width=\linewidth]{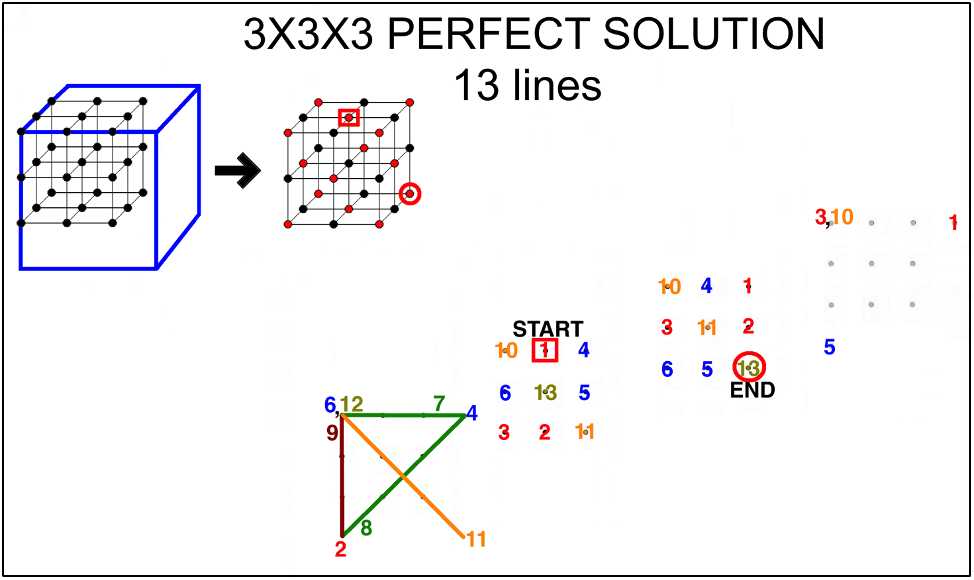}
\end{center}
\caption{$C(3)$ solves the $3 \times 3 \times 3$ puzzle inside a $3 \times 3 \times 3$ box ($27$ cubic units of volume), starting from face-centers or vertices, thanks to the clockwise-algorithm.}
\label{fig:Figure_3}
\end{figure}

Before moving on $k=4$, we wish to prove that the $3^3$-point problem is solvable starting from any node of $G_3$ if we exclude the center of the grid (as we have previously seen for \linebreak$k \in \{1,2\}$). This result immediately follows from symmetry when we combine the trails shown in Figures \ref{fig:Figure_3}\&\ref{fig:Figure_4}.

\begin{figure}[H]
\begin{center}
\includegraphics[width=\linewidth]{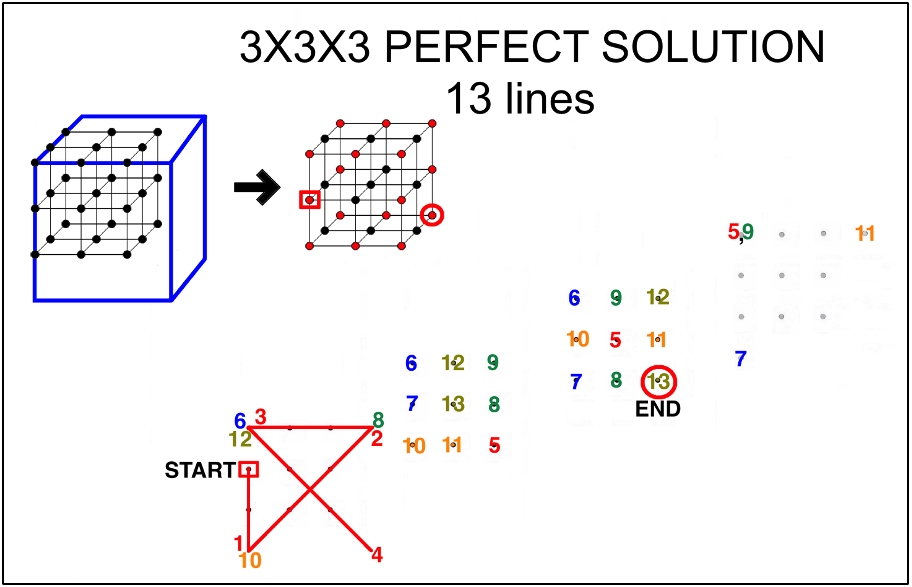}
\end{center}
\caption{Solving the $3 \times 3 \times 3$ puzzle inside a $3 \times 3 \times 3$ box ($27$ cubic units of volume), starting from edges or vertices.}
\label{fig:Figure_4}
\end{figure}

The number of solutions with $\frac{3^k-1}{2}$ lines increases as $k$ grows. Moreover, if we remove the box constraint, we can find new minimal covering trails \cite{13}, including those that reproduce (on a given $3 \times 3$ subgrid of $G_3$) the endpoints by Figure \ref{fig:Figure_2}, as shown in Figure \ref{fig:Figure_5}.

\begin{figure}[H]
\begin{center}
\includegraphics[width=\linewidth]{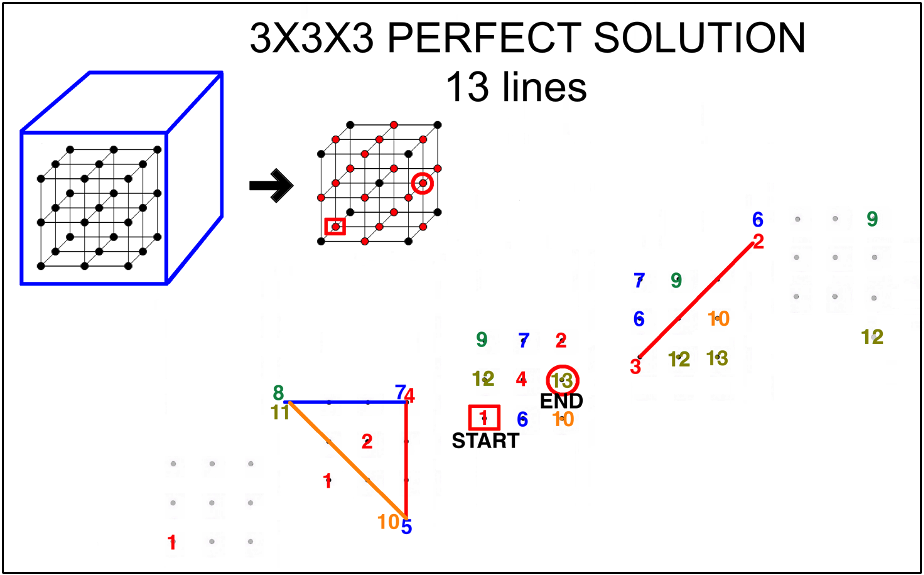}
\end{center}
\caption{Solving the $3 \times 3 \times 3$ puzzle inside a $3 \times 3 \times 4$ box ($36$ cubic units of volume).}
\label{fig:Figure_5}
\end{figure}

Finally, we present the solution to the $3^4$-point problem. Two examples of minimum length covering trails generated by the clockwise-algorithm are given.

The method to find $C(4)$ is basically the same one that we have previously discussed for $G_3$. So, we utilize the standard pattern shown in Figure \ref{fig:Figure_3} as we used $C(2)$ in order to solve the $3^3$-point problem. We apply $C(3)$ forward (while we spin around following the $3$-step gyratory as shown in Figure \ref{fig:Figure_6}), then backward (Figure \ref{fig:Figure_7}), subsequently we return to the starting vertex with line $27$ (the $(2 \cdot h(4-1)+1)$-th link), and lastly, we join the $3^3-1$ unvisited nodes with $C(3)$ by simply extending backward its first line (corresponding to the $28$-th link of $C(4)$ – see Figure \ref{fig:Figure_8}).

\begin{figure}[H]
\begin{center}
\includegraphics[width=\linewidth]{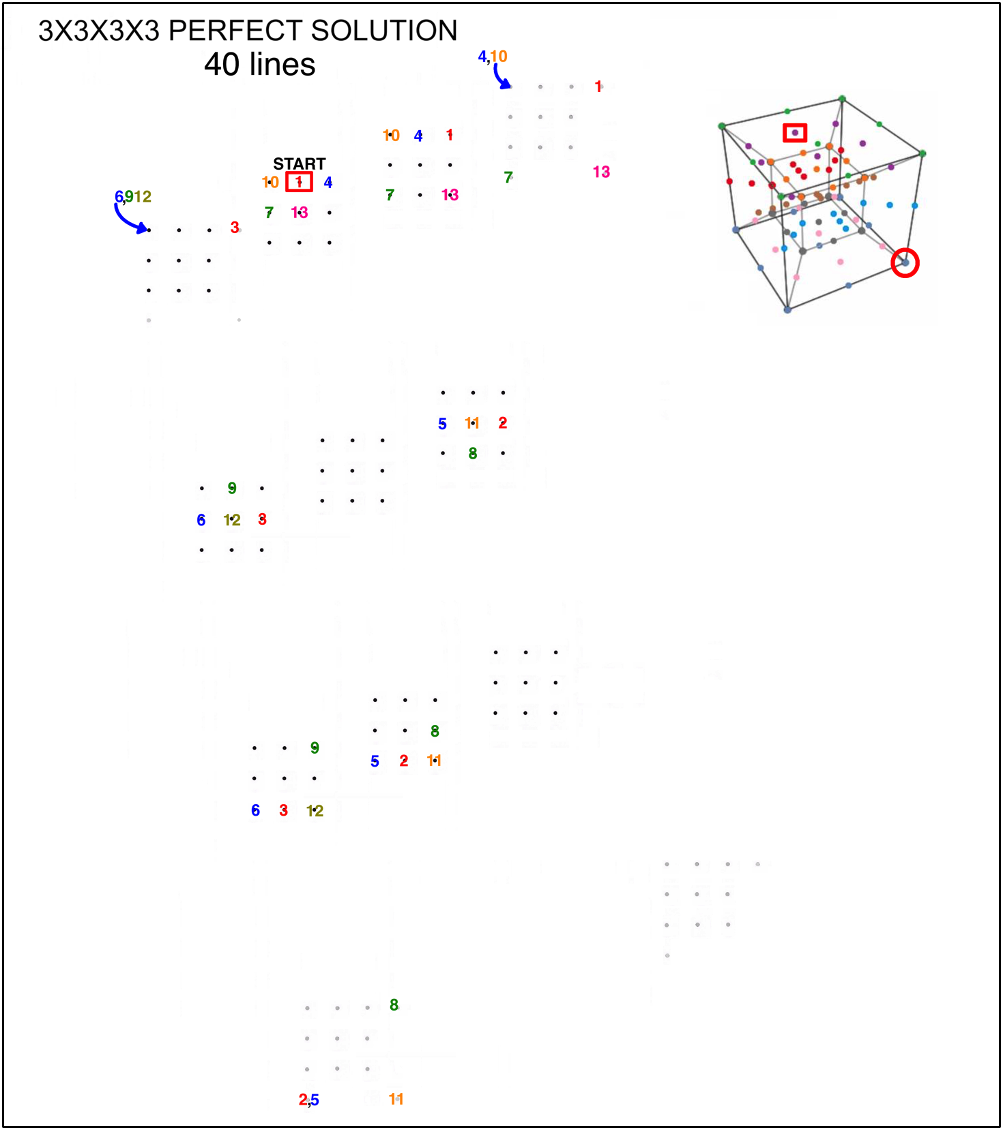}
\end{center}
\caption{Lines $1$ to $13$ of $C(4)$ following $C(3)$, as shown in Figure \ref{fig:Figure_3}.}
\label{fig:Figure_6}
\end{figure}

\begin{figure}[H]
\begin{center}
\includegraphics[width=\linewidth]{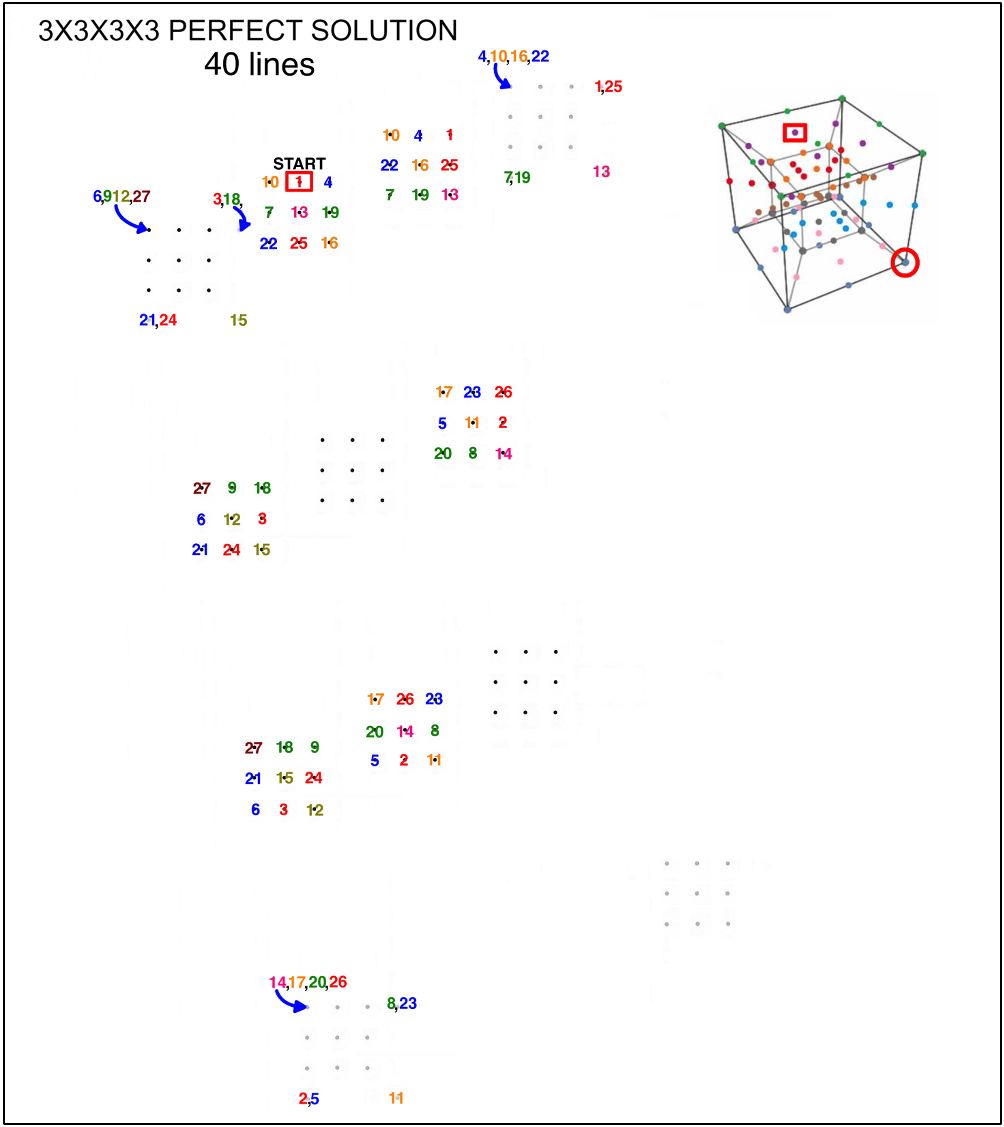}
\end{center}
\caption{Lines $14$ to $27$ of $C(4)$ following $C(3)$ backward, the $27$-th link to come back to the “starting point'' is also included.}
\label{fig:Figure_7}
\end{figure}

\begin{figure}[H]
\begin{center}
\includegraphics[width=\linewidth]{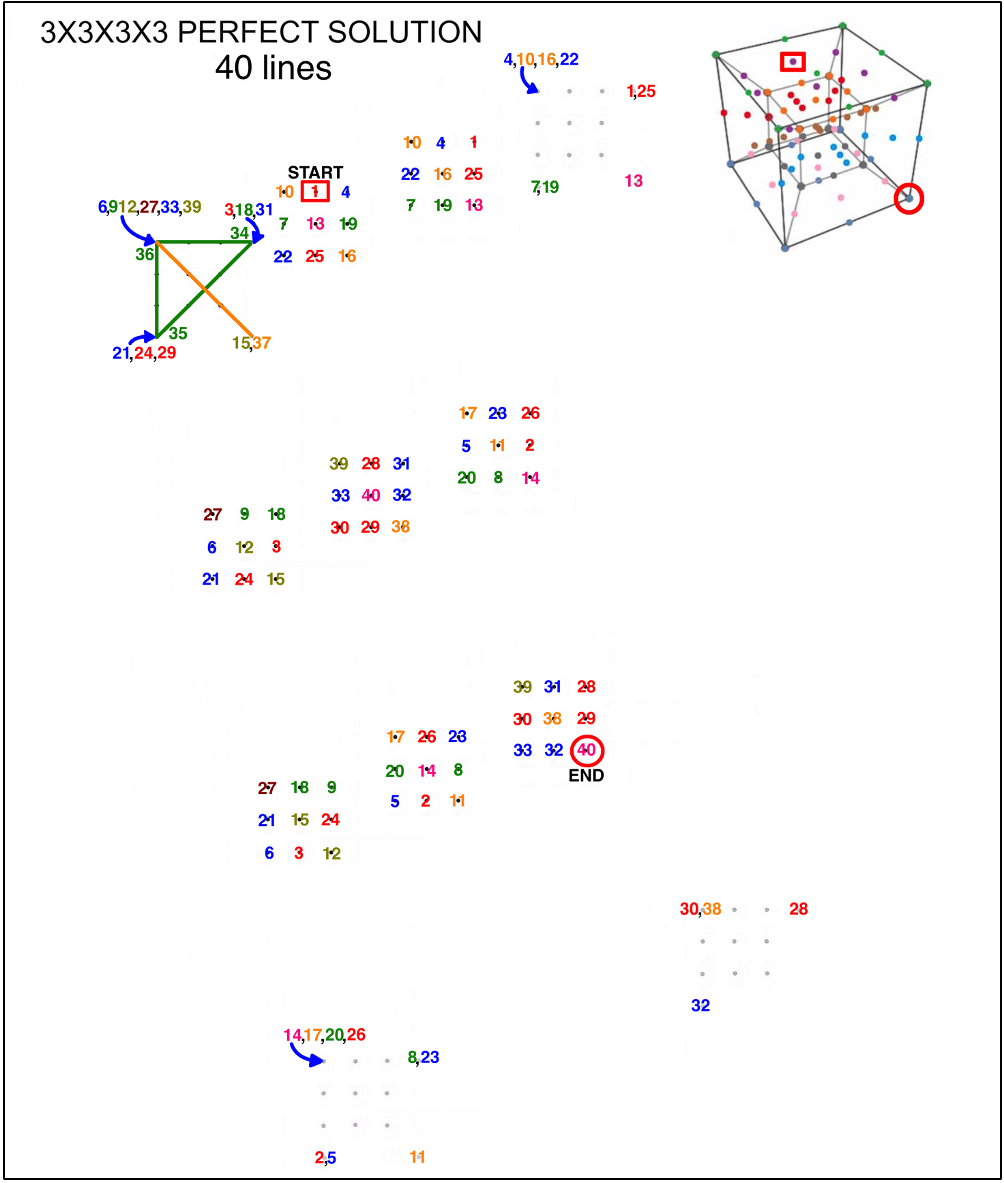}
\end{center}
\caption{A minimum length covering trail that completely solves the $3 \times 3 \times 3 \times 3$ puzzle with $40$ lines, inside a $3 \times 3 \times 3 \times 3$ box (hyper-volume $81$ units$^4$), thanks to the clockwise-algorithm applied to $C(3)$ from Figure \ref{fig:Figure_3}.}
\label{fig:Figure_8}
\end{figure}

The clockwise-algorithm reduces the complexity of the $3^k$-point problem to the complexity of the $3^{k-1}$-point one. A clear example is shown in Figure \ref{fig:Figure_9}.

\begin{figure}[H]
\begin{center}
\includegraphics[width=\linewidth]{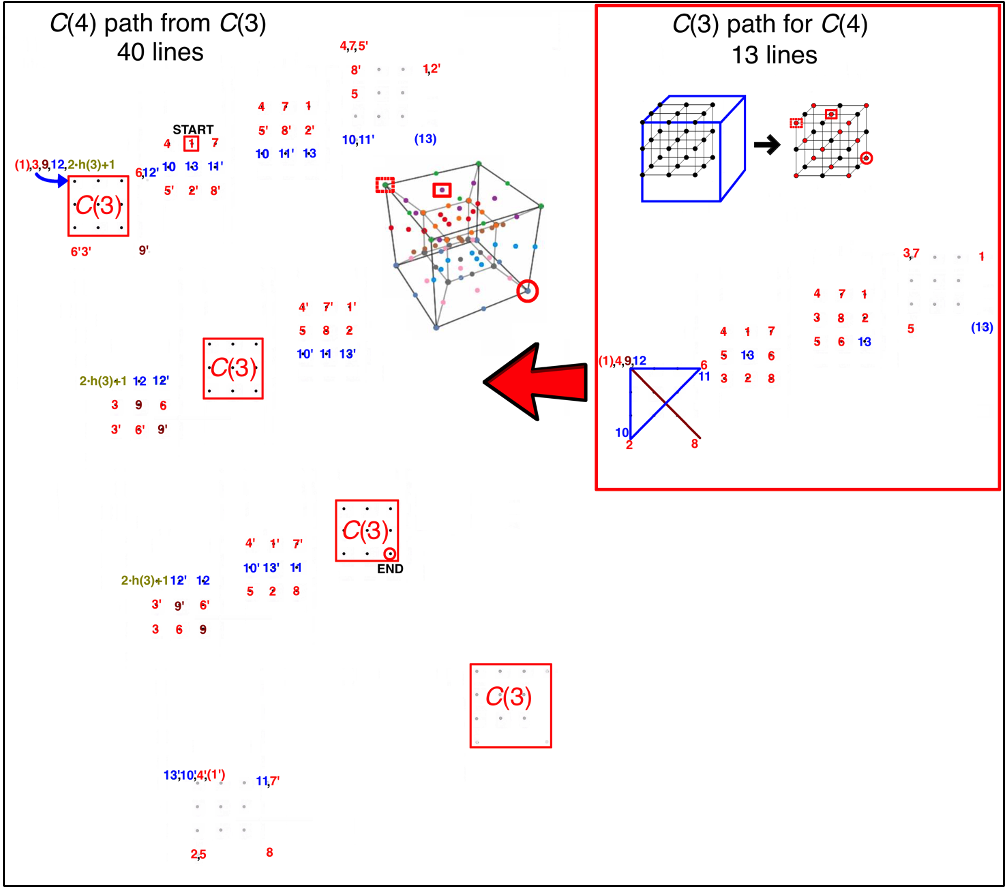}
\end{center}
\caption{How the clockwise-algorithm concretely works: it takes a minimum length covering trail $C(3)$ as input, and returns $C(4)$. Lines $1$-$13$ belong to the covering trail $C(3)$ (shown in the upper-right quadrant), line $13'$ follows line $13$ and belongs to $C(3)$ backward. $C(3)$ backward ends with line $1'$: it is extended (by one unit) in order to be connected to the $(2 \cdot h(3)+1)$-th link, and this allows $C(3)$ to be repeated
one more time (joining the remaining $26$ unvisited nodes).}
\label{fig:Figure_9}
\end{figure}

Since the clockwise-algorithm takes $C(k-1)$ as input and returns $C(k)$ as its output, it can be applied to any $C(k)$ in order to produce some $C(k+1)$ consisting of $h(k+1)=3 \cdot h(k)+1$ lines. Thus, it is possible to show by induction on $k$ that the $3^k$-point problem can be solved, inside a $3 \times 3 \times \cdots \times 3$ box of hyper-volume $3^k$ units$^k$, drawing optimal trails with $3 \cdot h(k-1)+1$ lines (Figure \ref{fig:Figure_10}).

Therefore, $\forall k \in \mathbb{N}-\{0\}$,
\begin{equation} \label{eq2}
h(k+1)=3 \cdot h(k)+1=\frac{3^{k+1}-1}{2}.
\end{equation}
\begin{figure}[H]
\begin{center}
\includegraphics[width=\linewidth]{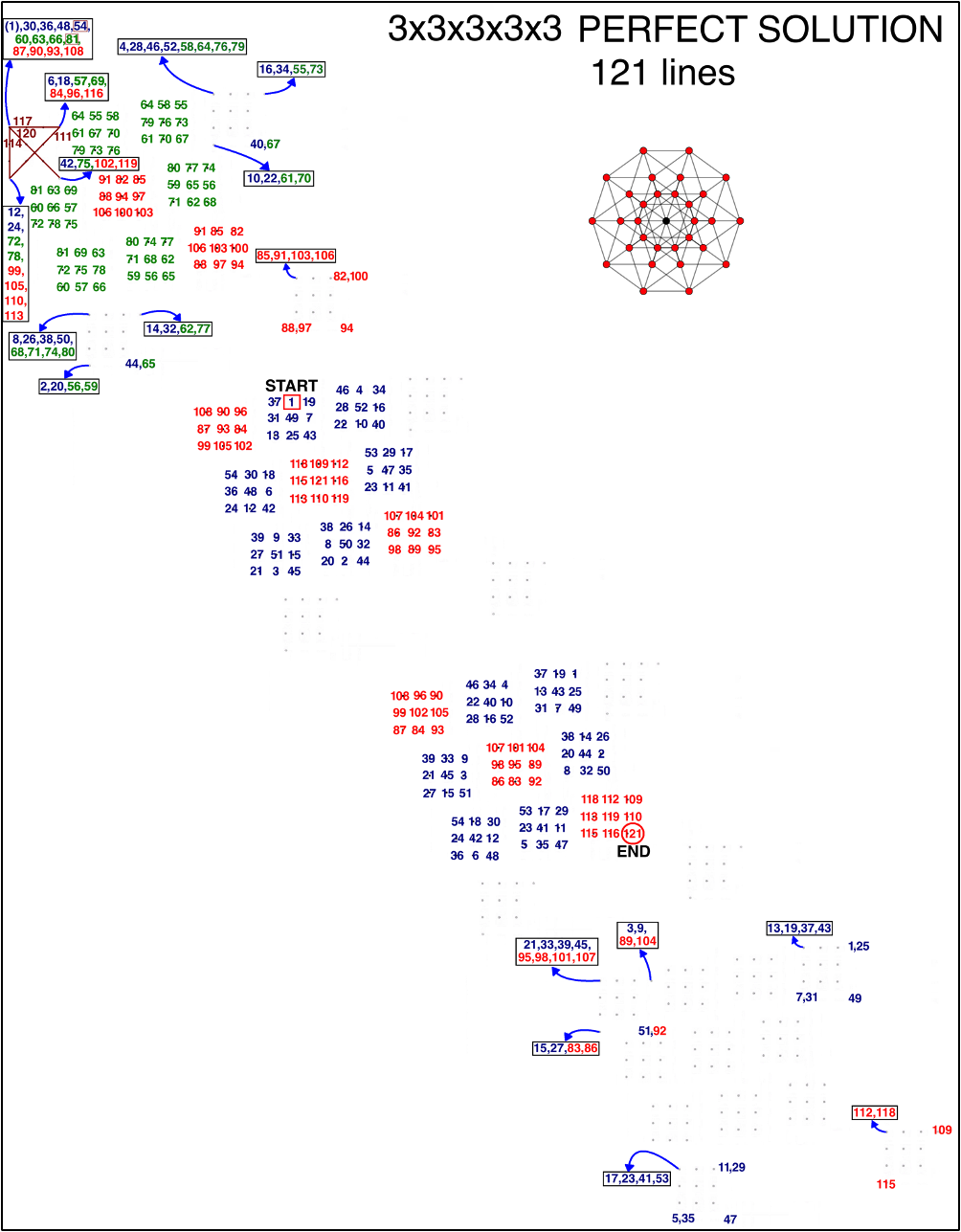}
\end{center}
\caption{For every integer $k$ greater than $1$, the $3^k$-point problem can be explicitly solved by the clockwise-algorithm ($k=5$ in our example). A $C(k)$ with $\frac{3^k-1}{2}$ lines immediately follows from any valid $C(k-1)$, and this surely occurs if $C(k-1)$ has one of its endpoints in a vertex of $G_{k-1}$.}
\label{fig:Figure_10}
\end{figure}


\section{Covering \texorpdfstring{$\bm{3^k}$ } -points by trees} \label{sec:3}

\begin{definition} \label{def2}
We call a tree any acyclic connected arrangement of line segments (i.e., edges of the tree) that covers some of the nodes of $G_k$, and we denote as $T(k)$ any tree (drawn in $\mathbb{R}^k$) that covers all the points belonging to the $k$-dimensional grid $G_k$. More specifically, $T(k)$ represents a covering tree for $G_k$ of size $t(k)$ (i.e., $T(k)$ has $t(k)$ edges).
\end{definition}

In 2014, Dumitrescu and T\'oth \cite{15} showed the existence of an inside the box covering tree for $G_k$, $\forall k \in \mathbb{N}-\{0\}$, of size $t_u(k)=h(k)=\frac{3^k-1}{2}$ (e.g., the set of all the endpoints of the $13$ edges of $t_u(3) \subset G_3$ – see Definition \ref{def1}). It is not hard to prove that, when we take as constraints our $3 \times 3 \times \cdots \times 3$ boxes (as usual), the upper bound $t_u(k)$ is not tight for every $k>3$.

\begin{lemma} \label{Lemma 1} Let $box \coloneqq \{(-1,0,1,2) \times (-1,0,1,2) \times \cdots \times (-1,0,1,2)\} \subset \mathbb{Z}^k$. For every $k$ greater than $3$, there exists a covering tree, $T(k)$, for $G_k$ whose all its vertices belong to box and such that $T(k)$ has size $t(k)<h(k)$.
\end{lemma}

\begin{proof}
We invoke Theorem \ref{Theorem 1} to remember that $h(k) \geq \frac{3^k-1}{2}$. It follows that it is sufficient to provide a general strategy to cover $G_k$ with a tree consisting of $\frac{3^k-1}{2}-c(k>3)$ edges, for some $c(k>3) \geq 1$. The tree in $\mathbb{R}^3$ shown in Figure \ref{fig:Figure_11}, which covers $3^3-1$ nodes of $G_3$ with its $12$ edges, also provides a valid upper bound for $t(4)$, since it is sufficient to clone twice the same pattern and spend one more link to join the remaining three collinear points belonging to each copy of $G_3$. So, we add $2$ more lines (at most) to connect every duplicated tree (to the other two copies of itself) and to fix the aforementioned link (which joins the last $3$ unvisited nodes of $G_4$), in order to create a covering tree of size $39$.

\begin{figure}[H]
\begin{center}
\includegraphics[width=\linewidth]{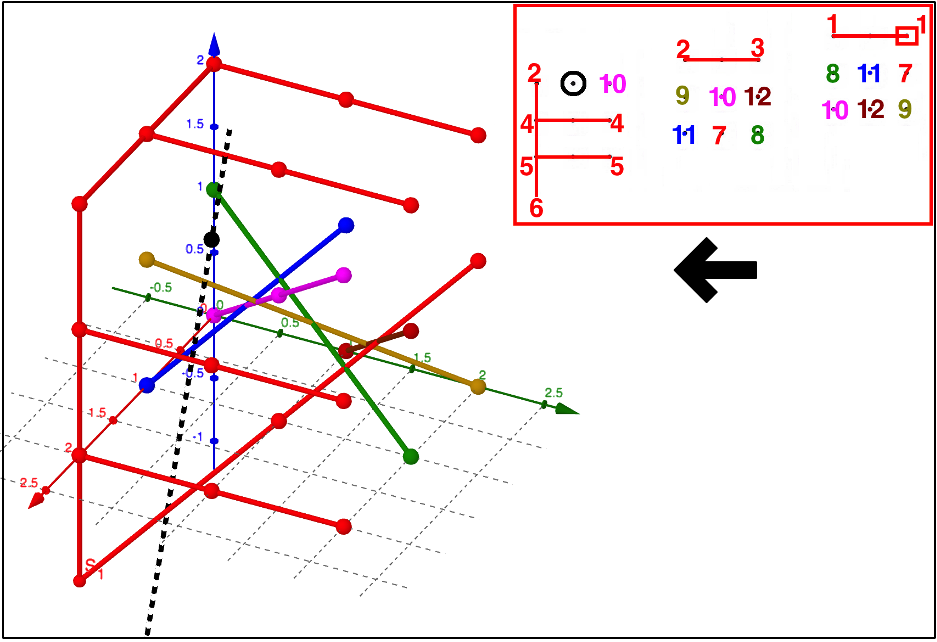}
\end{center}
\caption{An inside the ($2 \times 2 \times 3$) box tree with $t_u (3)-1=12$ edges that covers all the points of $G_3$ except the black one. The black dotted line represents the direction ($w$-axis) to fit the remaining three collinear points of $G_4$ when we replicate three times the same pattern \protect\cite{16}.}
\label{fig:Figure_11}
\end{figure}

Thus, we can generalize our result to all $k \geq 4$,
\begin{equation} \label{eq3}
t(k) \leq 3 \cdot t(k-1)+1 \leq 39 \cdot 3^{k-4}+\sum_{i=1}^{k-5} 3^i+1. 
\end{equation}

Hence,
\begin{equation} \label{eq4}
t(k) \leq \frac{3^{k-4}-1}{2}+13 \cdot 3^{k-3}.
\end{equation}

Therefore, $h(k)-t(k) \geq 3^{k-4} \geq 1$ holds for every $k \geq 4$.	
\end{proof} 

We are finally ready to remove the box constraint. Without any restriction to our \textit{thinking outside the box} ability, we are free to apply cleverly the idea introduced by Figure \ref{fig:Figure_11} to prove the existence of a covering tree for $G_3$ of size $t(3)=n^2+n$ (here $n$ assumes the odd value $3$ – see Reference \cite{15}, Section 3).

\begin{theorem} \label{Theorem 2}
The inequality $t(k)<h(k)$ holds if and only if $k \geq 3$.
\end{theorem}

\begin{proof}
Let $k=1$; it is trivial to verify that $t(1)=h(1)=1$.

If $k=2$, then $t(2)=h(2)=4$ (see Reference \cite{14}).

Thus, let $k=3$. Figure \ref{fig:Figure_12} shows the existence of a covering tree of size 
\begin{equation} \label{eq5}
12=t(3)<h(3)=13.
\end{equation}

\begin{figure}[H]
\begin{center}
\includegraphics[width=\linewidth]{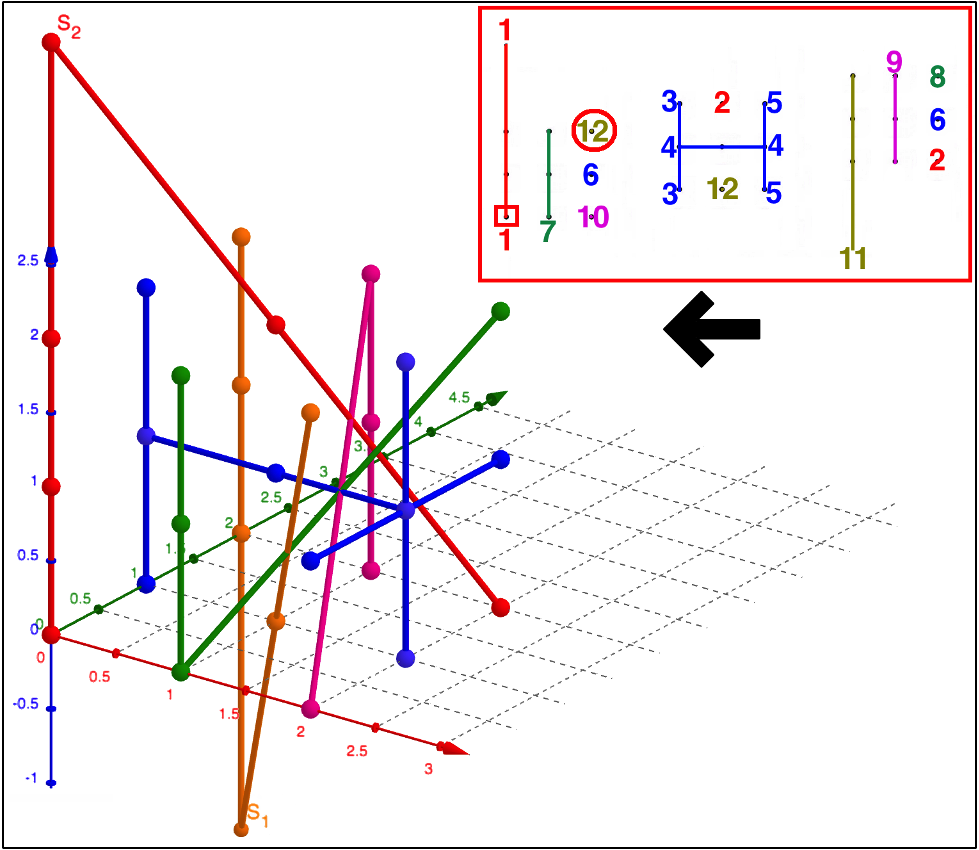}
\end{center}
\caption{A covering tree with $t(3)=12$ edges. $T(3)$ joins all the points of $G_3$ \protect\cite{16}.}
\label{fig:Figure_12}
\end{figure}

If $k \geq 4$, then Lemma \ref{Lemma 1} states that $t(k)<h(k)$. In particular, Equation (\ref{eq3}) shows that
\begin{equation} \label{eq6}
t(k) \leq (3 \cdot t(3)+1) \cdot 3^{k-4}\sum_{i=1}^{k-5} 3^i+1.
\end{equation}

Hence,
\begin{equation} \label{eq7}
t(k) \leq \frac{25 \cdot 3^{k-3}-1}{2}. 
\end{equation}

Since we already proven that $h(k)=\frac{3^k-1}{2}$ is optimal, it follows that
\begin{equation} \label{eq8}
h(k)-t(k) \geq \frac{3^k-1}{2}-\frac{25 \cdot 3^{k-3}-1}{2}. 
\end{equation}

Therefore, we conclude that $k \in \mathbb\{1,2\}$ implies $h(k)=t(k)$, whereas $h(k)-t(k) \geq 3^{k-3} \geq 1$ holds for every $k \in \mathbb{N}-\{0,1,2\}$.
\end{proof}


\section{Conclusion} \label{sec:Conc}

Given the $k$-dimensional grid $G_k$, the clockwise-algorithm lets us easily draw different covering trails with $\frac{3^k-1}{2}$ lines, and all of them remain inside the ($3 \times 3 \times \cdots \times 3$) box. After the $(3^k-1)$-th link, it is possible to switch from the previously applied $C(k-1)$ to another known solution of the $3^{k-1}$-point problem, completing a new optimal trial that has a different endpoint (e.g., we can take the walk shown in Figure \ref{fig:Figure_7} and then apply $C(3)$ from Figure \ref{fig:Figure_9}).

Let $X_k \equiv (1,1,\ldots,1)$ be the central node of $G_k$ (see Definition \ref{def1} for the case $k=3$).
\noindent We conjecture that, for every positive integer $k$, the $3^k$-point problem is solvable (embracing also every outside the box optimal trail) starting from each node of $G_k-\{X_k\}$ with a covering trail of length $h(k)=\frac{3^k-1}{2}$, while it is not if we include $X_k$ as an endpoint of $C(k)$.

\makeatletter
\renewcommand{\@biblabel}[1]{[#1]\hfill}
\makeatother

\end{document}